\documentclass[10pt]{amsart}
\usepackage{enumerate,amsmath,amssymb,latexsym,
amsfonts, amsthm, amscd, MnSymbol}

\theoremstyle{plain}

\newtheorem{theorem}{Theorem}

\numberwithin{equation}{section}

\newcommand{\C}{\mathbb{C}}

\newcommand{\E}{\mathbb{E}}

\newcommand{\PP}{\mathbb{P}}
\newcommand{\R}{\mathbb{R}}

\addtocounter{section}{-1}

\addtocounter{theorem}{-1}

\begin{document}

\title {A tangential approach to trigonometry}

\date{}

\author[P.L. Robinson]{P.L. Robinson}

\address{Department of Mathematics \\ University of Florida \\ Gainesville FL 32611  USA }

\email[]{paulr@ufl.edu}

 \keywords{}

\begin{abstract}

We construct the complex tangent as a meromorphic function in the plane, using an approach developed by Weierstrass in his characterization of analytic functions that satisfy algebraic addition theorems. 

\end{abstract}

\maketitle

\medbreak

\section*{Introduction} 

\medbreak 

Our purpose here is to set up trigonometry starting from the tangent function as solution to the familiar first-order initial value problem according to which $\phi ' = 1 + \phi^2$ and $\phi (0) = 0$. Were we to treat this as a real problem, we should of course not be able to extend the solution beyond the interval $(- \pi /2, \pi / 2)$ without imposing some such extra condition as periodicity; this being so, we regard the problem from a complex perspective, whereupon periodicity makes a natural appearance, the parameter $\pi$ entering of its own accord. 

\medbreak 

The standard existence-uniqueness theorem for initial value problems (due to Picard) is purely local; for reference, we shall state a convenient version of the theorem below. When applied to the initial value problem that introduces the tangent function, it only provides a unique holomorphic solution on some open disc about the origin. The problem of demonstrating that this local holomorphic solution extends to a global meromorphic solution is solved using a technique that was devised by Weierstrass in his determination of those analytic functions that satisfy algebraic addition theorems. 

\medbreak 

We pay close attention to the details of the construction. This enables us to recover the standard properties of the tangent function: its periodicity, its zeros, its poles, its nonvalues, its addition theorem and so on. We do not profess to have conducted the most efficient navigation of this particular route; indeed, it is possible to capture much of the structure with rather less attention to detail. Moreover, there are certainly more expeditious routes to the tangent function: not least expeditious is the traditional route that passes through the sine and cosine functions from the exponential function. What we have done is to construct the tangent function without the aid of any transcendental functions at all, our approach showing some of what can be achieved by use of the Identity Theorem (or the `principle of analytic continuation') in conjunction with the Picard theorem. 

\medbreak 

Once the tangent function is firmly established, the rest of trigonometry falls into place as a matter of course. The sine and cosine functions appear via rational functions of the tangent, as entire functions after the curing of removable singularities. The exponential function also appears in like manner, but this is left as an exercise. 

\medbreak 

We close this Introduction by stating a convenient version of the Picard theorem, as promised. See Chapter 2 in [Hille] for a proof and an extensive discussion of related results. 

\medbreak 

For first-order initial value problems in the complex domain, the standard `Picard' existence-uniqueness theorem (henceforth {\bf EU}) may be stated as follows. Consider the initial value problem 
\begin{eqnarray*}
w' & = & F(z, w) \\  w(z_0) & = & w_0
\end{eqnarray*} 
where $F(z, w)$ is holomorphic in the bidisk 
$$\{ (z, w) : |z - z_0| < a \; \; {\rm and} \; \; |w - w_0| < b \}$$
and there satisfies the inequalities 
$$|F(z, w)| \leqslant M$$
and 
$$|F(z, w') - F(z, w'')| \leqslant K |w' - w''|$$
\medbreak 
\noindent
for some positive constants $M$ and $K$. Define 
$$r = \min ( a, \; b/M).$$
Then there exists a unique holomorphic function 
$$f: B_r (z_0) \to \C$$
such that $f(z_0) = w_0$ and such that $f' (z) = F ( z, f(z) )$ whenever $|z - z_0| < r$. 

\medbreak 

\section*{The complex tangent function}

\medbreak 

We take the specific initial value problem (henceforth {\bf IVP}) 
\begin{eqnarray*}
w' & = & 1 + w^2 \\  w(0) & = & 0
\end{eqnarray*} 
as our starting point. The standard existence-uniqueness theorem {\bf EU} guarantees that {\bf IVP} has a unique solution in some open disc about $0$ as follows.

\medbreak 

\begin{theorem} \label{1/2}
 There exists a unique holomorphic function 
$$f : B_{1/2} (0) \to \C$$ 
that satisfies $f(0) = 0$ and satisfies $f' (z) = 1 + f(z)^2$ whenever $|z| < 1/2$. 
\end{theorem} 

\begin{proof} 
 With reference to {\bf EU}, here $F(z, w) = 1 + w^2$ has no explicit $z$ dependence, so the number $a$ is immaterial and can be arbitrarily large; once $b$ is chosen, we may take $1 + b^2$ as the bound $M$ on $|F|$ and may take $2 b$ as the Lipschitz constant $K$. Accordingly, {\bf IVP} has a unique solution in the open disc of radius $r = b/(1 + b^2)$ about $0$. The choice of $b$ is at our disposal: we choose $b = 1$ so that $b/(1 + b^2)$ assumes its maximum value $1/2$. 
\end{proof}

\medbreak 

Certain properties of this function are important to our development. Notice first of all that $f$ is odd: in fact, the function 
$$B_{1/2} (0) \to \C : z \mapsto - f(-z)$$ 
satisfies {\bf IVP} and so coincides with $f$ itself by the uniqueness in Theorem \ref{1/2}. Notice also that $f$ is `real' in the sense that if $z \in B_{1/2} (0)$ then $\overline{f(z)} = f(\bar{z})$: indeed, the function $B_{1/2} (0) \to C : z \mapsto \overline{f(\bar{z})}$ satisfies {\bf IVP} and so coincides with $f$ itself, again by Theorem \ref{1/2}; in particular, $f$ is real-valued on the real interval $(-1/2, 1/2)$. 

\medbreak 

The following result will be instrumental.  

\medbreak 

\begin{theorem} \label{Schw}
The function $f: B_{1/2} (0) \to \C$ takes its values in the open unit disc, whence if $z \in B_{1/2} (0)$ then $|f(z)| \leqslant 2 \: |z|$. 
\end{theorem} 

\begin{proof} 
With $M = 1 + b^2 = 2$ as in the proof of Theorem \ref{1/2}, if $z \in B_{1/2} (0)$ then
$$|f(z)| = | \int_0^z  f' \: | = | \int_0^z (1 + f^2) \: | \leqslant |z|  M = 2 |z| < 1$$
whence $|f(z)| \leqslant 2 \: |z|$ follows by the Schwarz Lemma. The fact that $f$ is valued in the open unit disc also follows at once from the proof of the Picard theorem. 
\end{proof} 

\medbreak 

We can determine precisely where $f$ vanishes.

\medbreak 

\begin{theorem} \label{zero}
The function $f: B_{1/2} (0) \to \C$ has $0$ for its only zero. 
\end{theorem} 

\begin{proof} 
If $z \in B_{1/2} (0)$ then  
$$f(z) = \int_0^z (1 + f^2) = z + \int_0^z f^2$$
whence by Theorem \ref{Schw} we deduce by integration along the complex interval $[0, z]$ that 
$$|f(z) - z \: | = |\int_0^1 f(t z)^2 z \: {\rm d}t| \leqslant |z| \int_0^1 (2 \: | t z|)^2 {\rm d} t = 4 \: |z|^3 /3$$
and conclude that if $f(z) = 0$ then $z = 0$. 
\end{proof} 

\medbreak 

We can also determine precisely where $f$ is real-valued. 

\medbreak 

\begin{theorem} \label{real}
The function $f: B_{1/2} (0) \to \C$ is real-valued on $\R \cap B_{1/2} (0)$ precisely.  
\end{theorem} 

\begin{proof} 
The fact that $f$ takes real values on $(-1/2, 1/2)$ was noted after Theorem \ref{1/2}. The fact that $f$ takes real values nowhere else may be established by modifying the proof of Theorem \ref{zero}. Explicitly, let $z = x + i y \in B_{1/2} (0)$ and join $0$ to $z$ by following first the real interval $[0, x]$ and then the complex interval $[x, x + i y]$. As $f$ is real-valued along $[0, x]$ we deduce that 
$${\rm Im} f(z) - y = {\rm Im} \int_x^{x + iy} f^2$$
and therefore by Theorem \ref{Schw} that 
$$|{\rm Im} f(z) - y | \leqslant |\int_0^1 f(x + i t y)^2  i y {\rm d} t| \leqslant |y| \int_0^1 |f(x + i t y)^2| {\rm d} t \leqslant |y| \int_0^1 4 (x^2 + t^2 y^2) {\rm d}t$$
whence certainly 
$$|{\rm Im} f(z) - y | \leqslant 4 (x^2 + y^2) |y|. $$
As $4(x^2 + y^2) < 1$ we conclude that if $f(z)$ is real then $y$ is zero. 
\end{proof} 

\medbreak 

Theorem \ref{1/2} is a firm step on the path to an independent construction of the complex tangent function, but it is only a beginning. We should like to extend the construction beyond the `small' disc $B_{1/2} (0)$ as far as is permissible. Of course, the Identity Theorem guarantees that any holomorphic (or meromorphic) extension of $f$ to a connected open set containing $B_{1/2} (0)$ is uniquely determined by $f$ itself. As matters stand at present, without prior knowledge of the complex tangent function, it is by no means clear how far $f$ may be extended; nor is it yet clear that analytic continuation does not lead to a multifunction. Knowing the tangent function as we do, we should like to be able to extend $f$ from $B_{1/2} (0)$ to a meromorphic function in the whole complex plane; but this extension must be accomplished directly by a close analysis involving the fundamental {\bf IVP}. 

\medbreak 

Full disclosure: we shall not be entirely forgetful regarding the properties of the tangent function; indeed, one of them - the duplication formula - suggests the method by which we shall extend $f$ to a meromorphic function in $\C$. This method, in essence due to Weierstrass and later adopted by Goursat, is a simplified version of one used by Neville in his elegant account of the Jacobian elliptic functions; see pages 505-506 of [Goursat] and pages 137-139 of [Neville]. 

\medbreak 

To be explicit, we start our construction with the function $f$ of Theorem \ref{1/2} and define the function 
$$F : B_1 (0) \to \C$$
 by the rule that if $z \in B_1 (0)$ then 
$$F (z) = \frac{2 f(\frac{1}{2} z)}{1 - f(\frac{1}{2} z)^2}\: . $$
This function $F$ is holomorphic, because $f$ assumes neither $1$ nor $-1$ as a value on $B_{1/2} (0)$ according to Theorem \ref{Schw}. Taking into account the fact that {\bf IVP} is satisfied by $f$, if $z \in B_1(0)$ then direct calculation yields 
$$ F' (z) = \frac{(1 + f(\frac{1}{2} z)^2)^2}{(1 - f(\frac{1}{2} z)^2)^2} = 1 + \frac{4 f(\frac{1}{2} z)^2}{(1 - f(\frac{1}{2} z)^2)^2} = 1 + F(z)^2$$
while also $F(0) = 0$ of course.  Thus the holomorphic function $F : B_1 (0) \to \C$ itself is a solution of {\bf IVP} and so agrees with the original $f$ on $B_{1/2} (0)$ by the uniqueness clause in Theorem \ref{1/2}: that is, $F |_{B_{1/2} (0)} = f$.

\medbreak 

 The temporary notation $F : B_1 (0) \to \C$ having served its clarifying purpose, we rename this extension as simply $f : B_1 (0) \to \C$. Notice that the conclusions of Theorem \ref{zero} and Theorem \ref{real} (suitably modified) continue to apply: the function $f : B_1 (0) \to \C$ vanishes only at $0$ and is real-valued precisely on $\R \cap B_1 (0)$. 

\medbreak 

\begin{theorem} \label{p}
The restriction of $f$ to the interval $(-1, 1)$ is strictly increasing and there exists a unique $p \in (0, 1)$ such that $f(\tfrac{1}{2} p ) = 1$. 
\end{theorem} 

\begin{proof} 
The strictly increasing nature of $f$ on $(-1, 1)$ is clear, because $f' = 1 + f^2 \geqslant 1 > 0$ there. Integration of this inequality shows first that if $t > 0$ then $f(t) > t$; repeated integration shows that if $t > 0$ then  $f(t) > t + \tfrac{1}{3} t^3$ whence $p$ exists as claimed. 
\end{proof} 

\medbreak 

We record this as a theorem simply for future reference. A consideration of the function inverse to $f|_{(-1, 1)}$ shows that 
$$\tfrac{1}{2} p = \int_0^1 \frac{{\rm d} t}{1 + t^2}. $$
Of course, outside knowledge confirms that the number $2 p$ is precisely $\pi$; indeed, the integral formula above is tantamount to a definition of $\pi$ that was adopted by Weierstrass. 

\medbreak 

Starting from this holomorphic solution $f : B_1 (0) \to \C$ to {\bf IVP} we consider the possibility of reduplication again. Thus, we define the holomorphic function 
$$F : B_2 (0) \setminus \{ \pm \: p \} \to \C$$ 
by the rule that if $z \in B_2 (0) \setminus \{ \pm \: p \}$ then 
$$F(z) = \frac{2 f(\frac{1}{2} z)}{1 - f(\frac{1}{2} z)^2} \ . $$
Differentiation again reveals that this function $F$ continues to satisfy {\bf IVP}; by uniqueness in Theorem \ref{1/2} it agrees with $f$ on $B_{1/2} (0)$ and by the Identity Theorem it agrees with $f$ on the unit disc $B_1(0)$; that is, $F|_{B_1 (0)} = f$. Note that the isolated singularity of $F$ at each of the points $\pm p$ is a (simple) pole, for there the numerator is nonzero while the denominator has a (simple) zero by Theorem \ref{p}; the residue of $F$ at each of these points is readily calculated to be $-1$. 

\medbreak 

Our analysis has identified a parameter that is naturally attached to the function $f$: namely, the positive real number $p$ of Theorem \ref{p}. The subsequent development is clarified considerably by taking this parameter into account. Thus, we shall restrict this new extension $F : B_2 (0) \to \C$ to the open disc $\E : = B_p (0)$ and shall denote by $f_0 : \E \to \C$ this new restriction. 

\medbreak 

Notice that $f_0 : \E \to \C$ is a holomorphic solution of {\bf IVP}; indeed, the disc $\E = B_p (0)$ is the largest disc about $0$ on which a holomorphic solution exists. Notice also that the conclusions of Theorem \ref{zero} and Theorem \ref{real} again continue to apply: the function $f_0 : \E \to \C$ vanishes just at $0$ and takes real values exactly on the interval $\R \cap \E = (- p, p)$. 

\medbreak 

Although the function $f_0 : \E \to \C$ was constructed merely by {\it duplication}, it actually satisfies a stronger {\it addition} theorem. 

\medbreak 

\begin{theorem} \label{add} 
If $a$ and $z$ lie in the disc $\frac{1}{2} \E = B_{\frac{1}{2} p} (0)$ then 
$$f_0 ( a + z) = \frac{f_0(a) + f_0(z)}{1 - f_0(a)f_0(z)} \: .$$
\end{theorem} 

\begin{proof} 
Notice that here, Theorem \ref{Schw} guarantees that the denominator does not vanish. With $a \in \frac{1}{2} \E$ fixed, consider the two holomorphic functions 
$$\tfrac{1}{2} \E \to \C : z \mapsto f_0(a + z)$$
and 

$$\tfrac{1}{2} \E \to \C : z \mapsto  \frac{f_0(a) + f_0(z)}{1 - f_0(a)f_0(z)} \: . $$
Direct calculation shows that each of these functions satisfies both the differential equation $\phi' = 1 + \phi^2$ and the initial condition $\phi (0) = f_0(a)$. According to {\bf EU} these functions agree near $0$; according to the Identity Theorem, they therefore agree on the whole disc $\tfrac{1}{2} \E$. 
\end{proof} 

We remark that by continuity, the same expression for $f_0(a + z)$ is valid when $a$ lies in the closed disc $\frac{1}{2} \overline{\E}$ and $z$ lies in the open disc $\frac{1}{2} \E$. In particular, from $f_0(\tfrac{1}{2} \: p) = 1$ and $f_0(\tfrac{1}{4} \: p) = \sqrt{2} - 1$ it follows that $f_0(\tfrac{3}{4} \: p) = \sqrt{2} + 1$ and therefore that $f_0(\tfrac{3}{2} \: p) = -1$. 

\medbreak 

Let us now return to duplication: two further applications and the process stabilizes, after which induction leads directly to the ultimate extension of $f$ as a meromorphic function in $\C$. 

\medbreak 

For the first of these further reduplications, we define the holomorphic function 
$$f_1: 2 \E \setminus  \{ \pm \: p \} \to \C$$
by the familiar rule that if $z \in 2\E \setminus  \{ \pm \: p \}$ then 
$$f_1(z) = \frac{2 f_0(\frac{1}{2} z)}{1 - f_0(\frac{1}{2} z)^2} \ . $$
This is again a solution to {\bf IVP}; by Theorem \ref{1/2} and the Identity Theorem, it agrees with $f_0$ on $\E$. As noted earlier, at each of the points $\pm p$ this extension $f_1$ has a simple pole of residue $-1$, for at these points the numerator is nonzero while the denominator has a simple zero. As before, $f_1$ is odd, vanishes at $0$ only and takes real values exactly on the real points in its domain. 

\medbreak 

The function $f _1: 2 \E \setminus \{ \pm \: p \} \to \C$ has the following partial periodicity.   

\medbreak 

\begin{theorem} \label{shift}
If $z \in B_{\tfrac{1}{2} p} (- \tfrac{1}{2} p)$ then $f_1(z + 2 p) = f_1(z)$. 
\end{theorem} 

\begin{proof} 
Consider two functions on the disc $B_{\tfrac{1}{2} p} (- \tfrac{1}{2} p)$: on the one hand, the restriction of $f_1$ itself; on the other hand, $g : B_{\tfrac{1}{2} p} (- \tfrac{1}{2} p) \to \C: z \mapsto f_1(z + 2 p)$. Each of these functions satisfies the same differential equation $\phi' = 1 + \phi^2$;  as remarked after Theorem \ref{add}, $f_1(\tfrac{3}{2} \: p) = -1$ so that $g(- \tfrac{1}{2} p)= f_1(- \tfrac{1}{2} p).$
It now follows by {\bf EU} and the Identity Theorem that $g$ agrees with $f_1$ on the disc $B_{\tfrac{1}{2} p} (- \tfrac{1}{2} p)$ and the proof is complete. 
\end{proof} 

\medbreak 

Of course, either by a parallel argument or by the fact that $f_1$ is odd, it is also true that if $|z - \tfrac{1}{2} p| < \tfrac{1}{2} p$ then $f_1(z - 2 p) = f_1(z)$. In particular, the behaviour of $f_1$ on the interval $(-p, 0)$ is repeated on the interval $(p, 2 p)$ and its behaviour on $(0, p)$ is repeated on $(-2 p, - p)$. More particularly still, $f (\pm \tfrac{1}{2} p) = \pm 1$ and $f(\pm \tfrac{3}{2} p) = \mp 1$. 

\medbreak 

We have reached a point at which it is appropriate to formulate precisely our inductive generation of the complex tangent function. Let us write $\PP$ for the set comprising precisely all integral multiples of $p$, writing $\PP_0$ for the even multiples of $p$ and $\PP_{\infty}$ for the odd multiples of $p$. When $n$ is any natural number, let us denote by $\E_n$ the open disc of radius $2^n p$ about $0$: thus, 
$$\E_n = 2^n \E = B_{2^n p} (0).$$ 

\medbreak 

Now, let $N$ be a positive integer and suppose that for each $1 \leqslant n \leqslant N$ we have constructed a holomorphic function 
$$f_n : \E_n \setminus \PP_{\infty} \to \C$$
so as to satisfy the following conditions: 
\medbreak 
\noindent 
$(\bullet)$ $f_n$ is a solution to {\bf IVP}; \\
$(\bullet)$ $f_n$ has simple poles precisely at the points of $\PP_{\infty}$ in $\E_n$; \\
$(\bullet)$ $f_n$ has simple zeros precisely at the points of $\PP_0$ in $\E_n$;  \\
$(\bullet)$ $f_n$ takes the value $\pm 1$ precisely at the points of $\tfrac{1}{2} \PP_{\infty}$ in $\E_n$;\\
$(\bullet)$ $f_n (z + 2^n p) = f_n (z)$ whenever $z + 2^{n - 2} p \in \E_{n - 2} \setminus \R$; \\
$(\bullet)$ $f_n$ agrees with $f_{n - 1}$ on $\E_{n - 1} \setminus \PP_{\infty}$; \\
$(\bullet)$ if $z \in \E_n \setminus (\PP_{\infty} \cup 2 \PP_{\infty})$ then 
$$f_n (z) = \frac{2 f_{n - 1} (\tfrac{1}{2} z)}{1 - f_{n - 1} (\tfrac{1}{2} z)^2}.$$

\medbreak 

The base step of this inductive construction is provided by the function $f_1$ that we fashioned from $f_0$ by duplication. With $f_N$ in hand, we construct  the next function 
$$f_{ N + 1} : \E_{N + 1} \setminus \PP_{\infty} \to \C$$
by reduplication as follows. We begin by defining $f_{N + 1}$ on $\E_{N + 1}\setminus (\PP_{\infty} \cup 2 \PP_{\infty})$ by the rule 
$$f_{N + 1} (z) = \frac{2 f_N (\tfrac{1}{2} z)}{1 - f_N (\tfrac{1}{2} z)^2}.$$ 
The resulting function is holomorphic on $\E_{N + 1}\setminus (\PP_{\infty} \cup 2 \PP_{\infty})$; let us inspect its behaviour on $2 \PP_{\infty}$ and on $\PP_{\infty}$. If here $z \in 2 \PP_{\infty}$ then $\tfrac{1}{2} z \in \PP_{\infty}$ so that (by induction) the numerator has a {\it simple} pole and the denominator has a {\it double} pole; consequently, the singularity of $f_{N + 1}$ (as defined thus far) is removable and when cured becomes a simple zero of the holomorphic function $f_{ N + 1} : \E_{N + 1} \setminus \PP_{\infty} \to \C$. If instead $z \in \PP_{\infty}$ then (by induction) the numerator is $\pm 2$ and the denominator is simply $0$; consequently, the singularity of $f_{N + 1}$ is a simple pole. This makes it clear that $f_{N + 1}$ has (simple) poles at the points of $\E_{N + 1} \cap \PP_{\infty}$ precisely. The point $z \in \E_{N + 1}$ is a zero of $f_{N + 1}$ exactly when $\tfrac{1}{2}$ is either a zero of $f_N$ (that is, $z \in 2 \PP_0$) or a pole of $f_N$ (that is, $z \in 2 \PP_{\infty}$). As $2 \PP_0 \cup 2 \PP_{\infty} = \PP_0$ this makes it clear that $f_{N + 1}$ has (simple) zeros at the points of $\E_{N + 1} \cap \PP_0$ precisely. Now let $z + 2^{N - 1} p \in \E_{N - 1} \setminus \R$: it follows that $\tfrac{1}{2} z + 2^{N - 2} p \in \E_{N - 2} \setminus \R$ and therefore (by induction) that 
$$f_{N + 1} (z + 2^{N - 1} p) = f_N ( \tfrac{1}{2} z + 2^{N - 2} p) = f_N (\tfrac{1}{2} z) = f_{N + 1} (z);$$
the Identity Theorem then ensures that the partial periodicity $f_{N + 1} (z + 2^{N - 1} p) = f_{N + 1} (z)$ actually holds whenever $z + 2^{N - 1} p \in \E_{N - 1} \setminus \PP_{\infty}$. Differentiation as before shows that the function $f_{N + 1} : \E_{N + 1} \setminus \PP_{\infty} \to \C$ satisfies {\bf IVP}; consequently, $f_{N + 1}$ agrees with  $f_N$ on $\E_N \setminus \PP_{\infty}$ by virtue of {\bf EU} and the Identity Theorem.  As $f_{N + 1}|_{\E_N \setminus \PP_{\infty}} = f_N$ takes values $\pm 1$ at the points of $\E_N \cap \tfrac{1}{2} \PP_{\infty}$ precisely, the partial periodicity just established guarantees that the odd function $f_{N + 1}$ takes the values $\pm 1$ alternately at the points of $\E_{N + 1} \cap \tfrac{1}{2} \PP_{\infty}$ precisely. This completes the inductive step in our construction. 
\medbreak 

We have arrived at our primary objective: the holomorphic function $f$ of Theorem \ref{1/2} extends into the plane as a meromorphic function for which we continue the symbol $f$. 

\medbreak 

\begin{theorem} \label{tangent}
The initial value problem {\bf IVP} admits a unique holomorphic solution 
$$f : \C \setminus \PP_{\infty} \to \C.$$
\end{theorem} 

\begin{proof} 
Existence is by mere collation of the sequence $(f_n : n \in \mathbb{N})$: when $z \in \C \setminus \PP$ define $f(z) : = f_n (z)$ for any $ n \in \mathbb{N}$ such that $n > |z|$; the choice of $n$ is immaterial, since each function in the sequence is the restriction of its successor. Uniqueness is plain. 
\end{proof} 

\medbreak 

This function $f$ is the complex tangent. Many of its familiar properties can be divined at once from the generating sequence $(f_n : n \in \mathbb{N})$. Of course, $f$ is both odd and `real'; further, it takes real values on $\R \setminus \PP_{\infty}$ precisely. Its pole set $\PP_{\infty}$ comprises precisely all odd multiples of the number $p$ specified in Theorem \ref{p}; moreover, each pole is simple and has residue $-1$. Its zeros are all simple and constitute the set $\PP_0$ comprising precisely all even multiples of $p$. It takes the values $\pm 1$ at precisely the half-odd-integer multiples of $p$: indeed, if $n \in \mathbb{N}$ then 
$$f ( (n + \tfrac{1}{2}) p ) = (- 1)^n.$$
\medbreak 

\medbreak 

The nonconstant meromorphic function $f$ has as many nonvalues as Picard allows. 

\medbreak 

\begin{theorem} \label{pm i}
The function $f$ has neither $i$ nor $-i$ as a value. 
\end{theorem} 

\begin{proof} 
Deny: assume that $f$ takes the value $i$ at $a \in \C \setminus \PP_{\infty}$. {\bf EU} provides a neighbourhood of $a$ on which the initial value problem `$\phi ' = 1 + \phi^2$ and $\phi (a) = i$' has a unique solution: the constant function $i$ being an obvious solution, it follows that $f$ is constantly $i$ near $a$ and hence throughout $\C \setminus \PP_{\infty}$ by the Identity Theorem; this is absurd.  
\end{proof} 

\medbreak 

The function $f$ satisfies an addition theorem. 

\medbreak 

\begin{theorem} \label{addition}
If $a, \; z, \; a + z$ lie in $\C \setminus \PP_{\infty}$ then
$$f ( a + z) = \frac{f(a) + f(z)}{1 - f(a)f(z)} \: .$$ 
\end{theorem} 

\begin{proof} 
Here, rather than use the result of Theorem \ref{add} it is perhaps easier to reproduce its proof. With $a \in \C \setminus \PP_{\infty}$ fixed, define functions $g_a$ and $G_a$ on their respective domains as follows. For $z$ not in the translate $\PP_{\infty} - a$ we define 
$$g_a (z) = f (a + z).$$ 
For $z$ neither in $\PP_{\infty}$ nor satisfying $f(a) f(z) = 1$ we define 
$$G_a (z) = \frac{f(a) + f(z)}{1 - f(a)f(z)} \: .$$
Note that each of these domains is connected, as is their intersection 
$$U_a = (\C \setminus \PP_{\infty}) \cap (\C \setminus (\PP_{\infty} - a)) \cap \{ z : f(a) f(z) \neq 1 \},$$
because the subset $ \{ z : f(a) f(z) = 1 \}$ of $\C \setminus \PP_{\infty}$ is discrete. By direct calculation each of the functions $g_a$ and $G_a$ satisfies the differential equation $\phi ' = 1 + \phi^2$ and the initial condition $\phi (0) = f(a)$. According to {\bf EU} they therefore agree on a neighbourhood of $0$; according to the Identity Theorem they therefore agree on their (connected) common domain $U_a$. 

This proves that the identity 
$$(1 - f(a) f(z)) f(a + z) = f(a) + f(z)$$
holds whenever $a, z, a + z \notin \PP_{\infty}$ and $f(a) f(z) \neq 1$. To complete the proof, we claim that if $a, z, a + z \notin \PP_{\infty}$ then the condition $f(a) f(z) \neq 1$ is in fact redundant, the hypothesis $f(a) f(z) = 1$ being untenable. To justify this claim, choose $r > 0$ so that $B_r (z) \subseteq (\C \setminus \PP_{\infty}) \cap (\C \setminus (\PP_{\infty} - a))$ and choose a sequence $(z_n : n \in \mathbb{N})$ in $B_r (z) \setminus \{z\}$ so that $z_n \to z$ and $f(a) f(z_n) \neq 1$; such a sequence exists, since otherwise the Identity Theorem would force constancy on $f$. Passage to the limit in the established relation 
$$(1 - f(a) f(z_n)) f(a + z_n) = f(a) + f(z_n)$$
then results in 
$$(1 - f(a) f(z)) f(a + z) = f(a) + f(z).$$
Here, by hypothesis, the left side vanishes, whence so does the right: the hypothesis $f(a) f(z) = 1$ then further forces $f(z)^2 = -1$ and thereby contradicts Theorem \ref{pm i}. 
\end{proof} 

\medbreak 

As noted in the proof, the assumption that $a + z, \; a$ and $z$ all lie in $\C \setminus \PP_{\infty}$ implies that the denominator in this addition formula is automatically nonzero. We leave as an exercise the task of deducing Theorem \ref{addition} from Theorem \ref{add} by application of the Identity Theorem. 

\medbreak 

The function $f$ is periodic, with identifiable periods. 

\medbreak 

\begin{theorem} \label{period}
The function $f : \C \setminus \PP_{\infty} \to \C$ has period-set precisely $\PP_0$. 
\end{theorem} 

\begin{proof} 
Otherwise said, the periods of $f$ coincide with its zeros. In the one direction, there is little to do: if $a $ is a period of $f$ then in particular $f(a + 0) = f(0) = 0$ so that $a \in \PP_0$. In the opposite direction we offer two proofs, taking (nonzero) $a \in \PP_0$ as given; note that $a$ is an even multiple of $p$ which we may inductively reduce to $2 p$. (1) If $z \in \C \setminus \PP_{\infty}$ then also $a + z \in \C \setminus \PP_{\infty}$ so that $f(a + z) = f(z)$ by virtue of the addition formula in Theorem \ref{addition}.  (2) The identity $f (z + 2 p) = f (z)$ is satisfied by all $z$ in the open disc of radius $\tfrac{1}{2} \: p$ about $-\tfrac{1}{2} \: p$ by Theorem \ref{shift}; by the Identity Theorem, it is satisfied by all $z$ in $\C \setminus \PP_{\infty}$ (period) 
\end{proof} 

\medbreak 

We remark that more is true: if $z, w \in \C \setminus \PP_{\infty}$ then $f(w) = f(z)$ precisely when $w - z \in \PP_0$; we leave the verification of this as an exercise. 

\medbreak 

At this point, we discontinue our close analysis of the tangent function itself. Naturally, there is much more that can be said; we mention only a couple of points, without dwelling on details. 

\medbreak 

The Weierstrass `duplication trick' is formulated by Neville in essentially the following terms: if $g$ is meromorphic in an open disc about $0$ and if $g( 2 z)$ can be expressed rationally in terms of $g(z)$ and $g'(z)$, then $g$ exists as a meromorphic function in the whole plane. It is certainly possible to extend the local holomorphic solution of {\bf IVP} to a global meromorphic solution of {\bf IVP} by applying duplication in this fundamental form. We choose to incorporate further information (including the location of zeros and poles) into the duplication process; proceeding thus, it is not necessary to establish the relevant properties of the tangent function separately, after its construction. Also regarding the `duplication trick', it of course rests on the elementary rational function 
$$\C \setminus \{ \pm 1 \} \to \C : w \mapsto \frac{2 w}{1 - w^2}$$
which has $\pm i$ as its fixed points. This function has real output precisely for real input; the fact that the complex tangent $f$ shares this property is built into our construction. We could have accorded a greater r\^ole to this elementary function in our analysis.   

\medbreak 

With the tangent function $f : \C \setminus \PP_{\infty} \to \C \setminus \{ \pm i \}$ in place as a foundation, the entire structure of trigonometry arises as expected. 

\medbreak 

As $f$ has neither $i$ nor $-i$ as a value, the function $1 + f^2$ is nowhere zero; we are therefore at liberty to define functions $c$ (for cosine) and $s$ (for sine) as follows. When $z \in \C \setminus 2 \PP_{\infty}$ we define 
$$c(z) = \frac{1 - f(\tfrac{1}{2} z)^2}{1 + f(\tfrac{1}{2} z)^2} \: .$$
The resulting function $c: \C \setminus 2 \PP_{\infty} \to \C$ is holomorphic and has each of the points in $2 \PP_{\infty}$ as a removable singularity with value $-1$; the corresponding extension $c : \C \to \C$ is thus entire. Likewise, we define the entire function $s : \C \to \C$ by first declaring that if $z \in \C \setminus 2 \PP_{\infty}$ then 
$$s(z) = \frac{2 f(\tfrac{1}{2} z)}{1 + f(\tfrac{1}{2} z)^2} $$
and then filling in the removable singularities at the points of $2 \PP_{\infty}$ as zeros (which they are, the numerator having a simple pole but the denominator having a double pole). Each of $c$ and $s$ is `real' and so takes real values at real points, this being true for $f$ itself. As regards parity, the oddness of $f$ forces $c$ to be even and $s$ to be odd, As regards zeros: $c(z) = 0$ precisely when $z \in \PP_{\infty}$ (precisely when $f(\tfrac{1}{2} z) = \pm 1$); $s(z) = 0$ precisely when $z \in \PP_0$ (precisely when $\tfrac{1}{2} z$ is either a pole or a zero of $f$). As regards periodicity: $f$ has period $2 p$, so each of $c$ and $s$ has $4 p$ as a period; of course, here $4 p = 2 \pi$. As regards addition formulae: the one for $f$ itself implies that if $a, z, a + z \in \C \setminus \PP_{\infty}$ then 
\begin{eqnarray*}
c(a + z) & = & \frac{(1 - f(a)f(z))^2 - (f(a) + f(z))^2}{(1 - f(a)f(z))^2 + (f(a) + f(z))^2} \\ & = & \frac{1 + f(a)^2 f(z)^2 - f(a)^2  - f(z)^2 - 4 f(a) f(z) }{(1 + f(a)^2)(1 + f(z)^2)} \\ & = & \frac{(1 - f(a)^2) (1 - f(z)^2) - 2 f(a) 2 f(z)}{(1 + f(a)^2)(1 + f(z)^2)} \\ & = & c(a) c(z) - s(a) s(z);
\end{eqnarray*}
it follows by the Identity Theorem (or indeed plain continuity) that if $a, z \in \C$ then   
$$c(a + z) = c(a) c(z) - s(a) s(z)$$ 
and similarly (or otherwise) 
$$s(a + z) = s(a)c(z) + c(a)s(z).$$

\bigbreak 

Returning to the setting of initial value problems, from the satisfaction of {\bf IVP} by $f$ we may directly derive the system 
\begin{eqnarray*}
c' & = & -s, \: \; c(0) = 1,\\ s' & = & + c, \; \; s(0) = 0.
\end{eqnarray*}

\medbreak 

Finally, it is perhaps worth mentioning that the relationship between the resulting second-order equation $g'' + g = 0$ (satisfied by the cosine and the sine) and the first-order equation $f' = 1 + f^2$ (satisfied by the tangent) is a prime example of the relationships that exist between linear second-order equations and Riccati equations.

\bigbreak

\begin{center} 
{\small R}{\footnotesize EFERENCES}
\end{center} 
\medbreak 

[Goursat] E. Goursat, {\it Cours d'Analyse Math\'ematique}, Volume II, Gauthier-Villars (1905).  

\medbreak

[Hille] E. Hille, {\it Ordinary Differential Equations in the Complex Domain}, Wiley-Interscience (1976); Dover Publications (1997).

\medbreak 

[Neville] E.H. Neville, {\it Jacobian Elliptic Functions}, Oxford University Press (1944). 

\medbreak

\medbreak

\end{document}